\documentclass[11pt,reqno]{amsart}

\usepackage{amsmath,amsthm,amssymb,mathtools,mathrsfs,color,tocvsec2}
\usepackage[tmargin=1in,bmargin=1in,rmargin=1in,lmargin=1in]{geometry}
\usepackage[breaklinks=true]{hyperref}

\newcommand{\edittwo}[1]{}

\theoremstyle{plain}

\newtheorem{theorem}[equation]{Theorem}
\newtheorem{corollary}[equation]{Corollary}
\newtheorem{proposition}[equation]{Proposition}
\newtheorem{lemma}[equation]{Lemma}

\theoremstyle{definition}
\newtheorem{definition}[equation]{Definition}
\newtheorem{remark}[equation]{Remark}
\newtheorem{exam}[equation]{Example}

\numberwithin{equation}{section}

\newcommand{\bx}{\mathbf{x}}
\newcommand{\by}{\mathbf{y}}
\newcommand\F{\mathbb{F}}
\newcommand\R{\mathbb{R}}
\newcommand\C{\mathbb{C}}
\newcommand\HH{\mathbb{H}}
\newcommand\G{\mathscr{G}}
\newcommand\norm{\mathcal{N}}
\DeclareMathOperator{\diam}{diam}

\begin{document}
\title{The non-compact normed space of norms on a finite-dimensional
Banach space}

\author{Apoorva Khare}
\address{Indian Institute of Science;
Analysis and Probability Research Group; Bangalore 560012, India}
\email{\tt khare@iisc.ac.in}


\begin{abstract}
We discuss a new pseudometric on the space of all norms on a
finite-dimensional vector space (or free module) $\mathbb{F}^k$, with
$\mathbb{F}$ the real, complex, or quaternion numbers. This metric arises
from the Lipschitz-equivalence of all norms on $\mathbb{F}^k$, and seems
to be unexplored in the literature. We initiate the study of the
associated quotient metric space, and show that it is complete,
connected, and non-compact. In particular, the new topology is strictly
coarser than that of the Banach--Mazur compactum.
For example, for each $k \geqslant 2$ the metric subspace $\{ \| \cdot
\|_p : p \in [1,\infty] \}$ maps isometrically and monotonically to $[0,
\log k]$ (or $[0,1]$ by scaling the norm), again unlike in the
Banach--Mazur compactum.

Our analysis goes through embedding the above quotient space into a
normed space, and reveals an implicit functorial construction of function
spaces with diameter norms (as well as a variant of the distortion). In
particular, we realize the above quotient space of norms as a normed
space.

We next study the parallel setting of the -- also hitherto unexplored --
metric space $\mathcal{S}([n])$ of all metrics on a finite set of $n$
elements, revealing the connection between log-distortion and diameter
norms. In particular, we show that $\mathcal{S}([n])$ is also a normed
space. We demonstrate embeddings of equivalence classes of finite metric
spaces (parallel to the Gromov--Hausdorff setting), as well as of
$\mathcal{S}([n-1])$, into $\mathcal{S}([n])$. We conclude by discussing
extensions to norms on an arbitrary Banach space and to discrete metrics
on any set, as well as some questions in both settings above.
\end{abstract}

\date{\today}
\maketitle

\settocdepth{section}
\tableofcontents

\section{The metric space of norms: definition and main result}

It is a folklore result that all norms on a finite-dimensional (real or
complex) normed linear space are topologically equivalent -- i.e.,
Lipschitz --  with respect to one another. The space of norms has long
been studied using the Banach--Mazur pseudometric.
Our goal in this work is to explain a new, strictly coarser topology on
the space of norms on $\R^k$ -- the equivalence classes are now given by
dilations -- which leads us to a \textit{non-compact} quotient metric
space $\mathcal{S}_k(\R)$. The Banach--Mazur continuum turns out to be a
(compact) quotient of this space; we will see for instance that the two
topologies agree on the sets of $p$-norms for $p \in [1,2]$ and
$[2,\infty]$, but not for $p \in [1,\infty]$.

We then study the space $\mathcal{S}_k(\R)$ by working in a broader
context of function spaces with diameter norms. As we explain below,
(a)~this function space construction is functorial and applies in a
special case to the setting of $\mathcal{S}_k(\R)$;
(b)~we deduce that the metric on $\mathcal{S}_k(\R)$ is in fact a norm;
and
(c)~we interpret this new metric/norm through a variant of the distortion
between metric spaces.
(d)~We also apply this functorial framework to deduce similar structural
properties of the metric space of all metrics on each finite set (see
Section~\ref{Sdistortion}), and of families of norms on an arbitrary
Banach space.
Hence the present paper, as we were surprisingly unable to find these
results recorded in the literature.\medskip

We begin by setting notation. Fix an integer $k>0$ and a Clifford algebra
$\F$ over $\R$ that is a division ring (equivalently, $\F$ lacks
zerodivisors), that is, $\F = \R, \C,$ or $\HH$.
We will denote $\dim_\R \F$ by $d$; also let $1,i$ (and $j,k$) denote the
standard $\R$-basis elements in $\C$ (or $\HH$).
Recall the conjugation operation $\alpha \mapsto \alpha^*$ in $\F$, which
is the unique $\R$-linear anti-involution that fixes $1$ and acts as
multiplication by $-1$ on $\{ i,j,k \} \cap \F$.
Now a \textit{norm} on $\F^k$ is a function $N : \F^k \to \R$ satisfying
the following properties for all $\bx, \by \in \F^k$ and $\alpha \in \F$:
\begin{enumerate}
\item \textit{Positivity:} $N(\bx) \geqslant 0$, with equality if and
only if $\bx = 0$.

\item \textit{Homogeneity:} $N(\alpha \bx) = |\alpha| N(\bx)$, where
$|\alpha| \coloneqq \sqrt{\alpha \alpha^*}$ will be termed the
\textit{absolute value} of $\alpha \in \F$ (to distinguish it from the
norm). Recall $|\cdot|$ is multiplicative on $\F$.

\item \textit{Sub-additivity:} $N(\bx+\by) \leqslant N(\bx) + N(\by)$.
\end{enumerate}

Denote the space of all norms on $\F^k$ by $\norm(\F^k)$. Here are some
basic properties of this space, some of which are used below.

\begin{lemma}\label{Lbasic}
For $\F = \R, \C,$ or $\HH,$ and an integer $k>0$, the space
$\norm(\F^k)$ is closed under the following operations:
\begin{itemize}
\item Addition.

\item Multiplication by $\R^{>0}$. (Thus, $\norm(\F^k)$ is a convex
cone.)

\item Pointwise limits, as long as the limiting function is positive
except at ${\bf 0}$.

\item Pre-composing by continuous additive maps $A : \F^k \to \F^k$ with
trivial kernel -- equivalently, real-linear maps $A \in GL_{dk}(\R)$
under some identification of $\F^k$ with $\R^{dk}$. In other words,
\[
A \in GL_{dk}(\R), \ N \in \norm(\F^k) \quad \implies \quad (\bx \mapsto
N(A \bx)) \in \norm(\F^k).
\]
\end{itemize}
\end{lemma}

Notice that there are also other ways to construct norms, e.g. adding
norms on subspaces of $\F^k$ to a given norm in $\norm(\F^k)$. See
Equation~\eqref{Eaxes} below for an example.

The next result is standard for $\F = \R$, and easily extends to $\C$ or
$\HH$.

\begin{lemma}\label{Lequiv}
All norms in $\norm(\F^k)$ are Lipschitz-equivalent, i.e., for any two
norms $N,N' \in \norm(\F^k)$ there exist constants $0 < m \leqslant M$
such that
\begin{equation}\label{Elip}
m \cdot N(\bx) \leqslant N'(\bx) \leqslant M \cdot N(\bx), \qquad \forall
\bx \in \F^k.
\end{equation}
\end{lemma}

For completeness we include a proof-sketch, in a slightly more general
setting that is relevant to the present work (below).

\begin{proof}
This follows from two observations -- (i)~$\F^k$ is a finite-dimensional
vector space over $\R$, and (ii)~$\norm(\F^k) \subset \norm(\R^{dk})$
under the $\R$-linear homeomorphism $\F \cong \R^d$, where $d = \dim_\R
\F$. These observations reduce the situation to the well-known case of
$\F = \R$, where we remind that the result again follows from two
observations:
(a)~Every norm is bounded above by a positive multiple of the sup-norm;
one obtains this by working on the boundary $\partial C$ of the cube $C =
[-1,1]^d$.
(b)~Given a compact metric space $(X,d)$ (such as $X = \partial C$), all
continuous maps in $C(X, (0,\infty))$ are pairwise `Lipschitz
equivalent'. A more general statement is that given any set $X$, all set
maps $: X \to (0,\infty)$ with image bounded away from $0$ and $\infty$
are pairwise `Lipschitz equivalent'.
\end{proof}

The preceding result is well-known. A less well-known result (which we
were unable to find in the literature) is the following construction,
which was mentioned by V.G.~Drinfeld in a lecture at the University of
Chicago in the early 2000s:

\begin{proposition}\label{Pdrinfeld}
Say that two norms $N,N'$ on $\F^k$ are {\em equivalent}, written $N \sim
N'$, if $N' \equiv \alpha N$ for some positive real number $\alpha$. Then
the space $\mathcal{S}_k(\R) \coloneqq \norm(\R^k) / \sim$ is a metric
space, with metric
\begin{equation}\label{Emetric}
d_{\mathcal{S}_k(\R)}([N], [N']) \coloneqq \log (M_{N,N'} / m_{N,N'}),
\end{equation}

\noindent where $M = M_{N,N'}$, $m = m_{N,N'}$ denote the largest and
smallest Lipschitz constants respectively, in Equation \eqref{Elip}.
\end{proposition}

Once formulated, the result is shown in a straightforward manner.
Informally, `the space of metrics forms a metric space'. The reader may
also recognize~\eqref{Emetric} as a variant of the
\textit{log-distortion} between metrics; see Section~\ref{Sdistortion}
for more on this.

\begin{remark}
We now record the connection between the space $\mathcal{S}_k(\F)$ and
the Banach--Mazur compactum (see e.g.~\cite{Pi}), where two
$k$-dimensional Banach spaces $U,V$ over $\F = \R$ or $\C$ have distance
\[
\log \inf \{ \| T \| \cdot \| T^{-1} \| : T \in GL(U,V) \}.
\]
Now if two norms are proportional, and thus represent the same point in
$\mathcal{S}_k(\R)$, then they also do the same in the Banach--Mazur
compactum: just note that $\| T \| \cdot \| T^{-1} \| = 1$ for $T$ the
identity map on $\R^k$. It follows that the Banach--Mazur compactum is a
quotient of $\mathcal{S}_k(\R)$. One consequence of our main result
(Theorem~\ref{Tmain} below) is that the topology in $\mathcal{S}_k(\R)$
is strictly coarser.
\end{remark}

The space $\mathcal{S}_k(\F)$ does not seem to be known to experts, nor
is it defined or analyzed in the literature; we initiate its study in the
present work. In light of the preceding remark, we hope that subsequent,
continued analysis of $\mathcal{S}_k(\F)$ will also yield additional
information about the Banach--Mazur compactum.

We begin with an immediate consequence of the above observation that
$\norm(\F^k) \subset \norm(\R^{dk})$: in a sense, it suffices to work
with $\F = \R$ (as we do below):

\begin{corollary}\label{C18}
The space $\mathcal{S}_k(\F) \coloneqq \norm(\F^k) / \sim$ is a closed
metric subspace of $\mathcal{S}_{dk}(\R)$, with common metric given
by~\eqref{Emetric}.
\end{corollary}

\begin{proof}
We show that $\mathcal{S}_k(\F)$ is closed in $\mathcal{S}_{dk}(\R)$.
Suppose $[N_l] \to [N]$ in $\mathcal{S}_{dk}(\R)$, with $N_l \in
\norm(\F^k)\ \forall l$ and $N \in \norm(\R^{dk})$. Without loss of
generality, rescale the $N_l$ and assume via \eqref{Elip} that
\[
\frac{N_l}{N} : \F^k \setminus \{ {\bf 0} \} \to [1,M_l], \qquad \forall
l>0
\]

\noindent with $M_l \to 1$ as $l \to \infty$. But then, $N_l \to N$
pointwise on $\F^k$. In particular, given $\alpha \in \F$ and nonzero
$\bx \in \F^k$,
\[
\frac{N(\alpha \bx)}{N(\bx)} = \lim_{l \to \infty} \frac{N_l(\alpha
\bx)}{N_l(\bx)} = |\alpha|,
\]
and from this it follows that $N \in \norm(\F^k)$ as desired.
\end{proof}

We now state the main result of the present work (with the caveat that
this result is placed in a more general, functorial framework introduced
in the following section). It implies as a consequence that
$d_{\mathcal{S}_k(\F)}$ is not just a metric, but also a norm:

\begin{theorem}\label{Tmain}
For $\F = \R, \C, \HH$, the space $\mathcal{S}_k(\F)$ is a complete,
path-connected metric subspace of a real Banach space. It is a singleton
set for $k=1$, and unbounded for $k>1$.
\end{theorem}

\noindent (In less formal terms: `the space of norms lies in a normed
linear space.') A second consequence is that in dimensions two and
higher, the space of equivalence classes of norms is not compact.

\subsection*{Acknowledgments}

This work is partially supported by Ramanujan Fellowship
SB/S2/RJN-121/2017 and MATRICS grant MTR/2017/000295 from SERB (Govt.~of
India), by grant F.510/25/CAS-II/2018(SAP-I) from UGC (Govt.~of India),
and by a Young Investigator Award from the Infosys Foundation. I thank
Terence Tao for valuable discussions -- especially about the final
section -- as well as Gautam Bharali, Javier Cabello S\'anchez, Hariharan
Narayanan, and M. Amin Sofi for several useful comments that helped
improve previous versions of this manuscript. Part of this work was
carried out during a visit to UCLA, to which I am thankful for its
hospitality.

\section{Diameter norms and an endofunctor}\label{S2}

The goal of this section and the next is (to proceed toward) proving
Theorem~\ref{Tmain}. While it is possible to provide a direct proof, our
construction of a family of Banach spaces that each encompass
$\mathcal{S}_k(\F)$ (as asserted in Theorem~\ref{Tmain}) turns out to be
part of a broader functorial setting -- which we will use below in more
than one setting. Thus we explain this setting in the present section,
and complete the proof in Section~\ref{S3}.

The most primitive framework we consider is that of an abelian
topological semigroup $(\G,+,d_\G)$ with an associative, commutative
binary operation $+ : \G \times \G \to \G$ and a translation-invariant
metric $d_\G$, i.e.,
\begin{equation}\label{Esemimetric}
d_\G(x+z,y+z) = d_\G(x,y), \qquad \forall x,y,z \in \G.
\end{equation}

Notice that in such a semigroup, one does not necessarily have inverses
(i.e., `negatives') or the identity element $e = 0_\G$. However, the
following is easily shown (see e.g.~\cite{KR} for details).
\begin{itemize}
\item the semigroup always has at most one idempotent $2e=e$;
\item if such an $e$ exists, it is the unique identity element in (the
monoid) $\G$;
\item if $\G$ does not contain an idempotent, one can formally attach
such an idempotent $e$ with metric $d_\G(e,z) \coloneqq d_\G(z,2z),
d_\G(e,e) = 0$, and this creates the unique smallest monoid (with
translation-invariant metric) containing $\G$.
\end{itemize}

Examples of such semigroups $\G$ abound in the literature, the most
prominent being Banach spaces. However, there are several `intermediate'
classes of such abelian semigroups, including
monoids (i.e.~`$\mathbb{N} \cup \{ 0 \}$-modules'),
groups (i.e.~$\mathbb{Z}$-modules),
torsion-free divisible groups (i.e.~$\mathbb{Q}$-modules),
and normed linear spaces (i.e.~$\R$-modules).
More generally, one can consider metric $R$-modules, where $R \subset \R$
is a unital subring. As a further variant, one has the subclass of
$R$-normed $R$-modules (see Definition~\ref{Dcat}).

Our first goal is to show that the seminorm defined in
Theorem~\ref{Tsemigroup} below endows all $\G$-valued function spaces
with the structure of a (pseudo-)metric. In fact, we show that this holds
on a more structural level. For instance, it is clear that if $\G$ is a
monoid, then so is the corresponding function space; and this holds for
the finer structures mentioned above as well. What we also show is that
the function space construction is also compatible with `good'
homomorphisms.

A systematic way to carry out this bookkeeping is that each of the above
classes of semigroups is in fact a category, and the function space
construction is a \textit{covariant endofunctor} for each of these
categories. This is now explained.

\begin{definition}\label{Dcat}
Let $R \subset \R$ denote any unital subring.
\begin{enumerate}
\item Let ${\tt Semi}$ denote the category of abelian topological
semigroups $(\G, +, d_\G)$ with translation-invariant metric $d_\G$, as
above, and whose morphisms are semigroup homomorphisms that are
Lipschitz.

\item Let ${\tt Mon}$ denote the full subcategory of ${\tt Semi}$, whose
objects are monoids.

\item Let $R$-${\tt Mod}$ denote the subcategory of ${\tt Semi}$, whose
objects are $R$-modules such that multiplication by scalars in $R$ are
Lipschitz maps, and whose morphisms are Lipschitz $R$-module maps.

\item Let $R$-${\tt NMod}$ denote the full subcategory of $R$-${\tt
Mod}$, whose objects are \textit{$R$-normed} $R$-modules $\G$. In other
words, $d_\G(0, ra) = |r| d_\G(0,a)$ for all $r \in R$ and $a \in \G$.

\item Let $\overline{R}$-${\tt Mod} \subset R$-${\tt Mod}$ (note the
unconventional notation) denote the full subcategory of $R$-${\tt Mod}$
whose objects are complete metric spaces; and similarly define
$\overline{R}$-${\tt NMod} \subset R$-${\tt NMod}$.
\end{enumerate}
\end{definition}

For each of these classes of objects, we now prove that the following
function space construction is functorial:

\begin{theorem}\label{Tsemigroup}
Suppose $\mathscr{C}$ is one of the categories in Definition \ref{Dcat}
(i.e., ${\tt Semi}, {\tt Mon}, R$-${\tt Mod}, \dots, \overline{R}$-${\tt
NMod}$).
Given a set $X$ and an object $M \in \mathscr{C}$, let $F_b(X,M)$ denote
the set of bounded functions $f : X \to M$, and define
\begin{equation}\label{Esemigroup}
d(f,g) \coloneqq \sup_{x,x' \in X} d_M(f(x) + g(x'), f(x') + g(x)).
\end{equation}
For every set $X$ that is not a singleton:
\begin{enumerate}
\item $d$ is a pseudometric on $F_b(X,M)$.

\item If one defines $f \sim g$ to mean $d(f,g) = 0$, then $\sim$ is an
equivalence relation, and the quotient space assignment $M \mapsto
F_b(X,M) / \sim$ is a covariant isometric endofunctor of the category
$\mathscr{C}$.
\end{enumerate}
Here, we term a functor $F : \mathscr{C} \to \mathscr{C}'$ to be {\em
isometric} if:
(a)~all ${\rm Hom}$-spaces in $\mathscr{C}, \mathscr{C}'$ are metric
spaces, and
(b)~$F : {\rm Hom}(C_1,C_2) \to {\rm Hom}(F(C_1), F(C_2))$ is an isometry
for all objects $C_1, C_2 \in \mathscr{C}$.
\end{theorem}

\begin{remark}
As we show in Section~\ref{S3}, the connection to norms arises from the
fact that the space $\mathcal{S}_k(\F) \subset \mathcal{S}_{dk}(\R)$
embeds into $F_b(X,\R) / \sim$ for some compact subset $X \subset \F^k$.
Thus by Theorem~\ref{Tsemigroup} for $\mathscr{C} = \overline{\R}$-${\tt
NMod}$, the metric on the space of norms arises as a norm in a Banach
space.
\end{remark}

For more categorical consequences and ramifications related to
Theorem~\ref{Tsemigroup}, we refer the reader to e.g.~\cite{KR}. 
\edittwo{Also check here -- or already done in \cite{KR}? -- whether
or not Hom-spaces with values in $\mathscr{C}$-objects, are themselves
objects of $\mathscr{C}$?}
Also notice that if $X$ is a singleton then $F_b(X,M)
\simeq M$, whence $F_b(X,M) / \sim$ is the trivial semigroup (or Banach
space). In this case the above result is true, except perhaps for the
word `isometric'.

\begin{remark}\label{Rdiameter}
For general abelian metric semigroups $M \in \mathscr{C}$ as in
Theorem~\ref{Tsemigroup}, an example of functions $f,g \in F_b(X,M)$ with
distance zero is to choose and fix $m_0 \in M$, and take $g(x) \equiv m_0
+ f(x)$ on all of $X$. If $M$ is moreover a monoid and $f \equiv 0_M$,
then these are the only examples.

In the further special case when $M$ is an abelian metric group, the
functions $g \equiv m_0 + f$ turn out to be the only examples of
equivalent functions, for all $f \in F_b(X,M)$. Moreover, the
(pseudo-)metric defined above has a more accessible interpretation as a
\textit{diameter seminorm}:
\begin{equation}
d(f,g) = N(f-g), \quad \text{where} \quad N(f) \coloneqq \sup_{x,x' \in
X} d_M(f(x), f(x')) = {\rm diam} ({\rm im}(f)).
\end{equation}
In particular, the equivalence relation $f \sim g$ amounts to $f-g$ being
a constant function. For (abelian) monoids $(M, 0_M)$, essentially this
last assertion also follows if one restricts to the set of bounded
functions $f : X \to M$ satisfying: $0_M \in {\rm im}(f)$. More
precisely, if $d(f,g) = 0$ for such functions, there exists $m_0 \in M$
such that $-m_0 \in M$ and $g \equiv m_0 + f$ on $M$.
\end{remark}

\begin{proof}[Proof of Theorem \ref{Tsemigroup}]
We begin by showing that (1) holds for all abelian semigroups $M$, and
then turn to (2) for each successively smaller category. To show (1), we
will show the triangle inequality; note this also proves the transitivity
of $\sim$ and hence that $\sim$ is an equivalence relation on $F_b(X,M)$.
Given $x,y \in X$,
\begin{align}\label{Etriangle}
&\ d_M(f(x) + g(y), f(y) + g(x)) \notag\\
= &\ d_M(f(x) + g(y) + h(y), f(y) + g(x) + h(y)) \notag\\
\leqslant &\ d_M(f(x) + g(y) + h(y), f(y) + g(y) + h(x))\\
& \quad + d_M(f(y) + g(y) + h(x), f(y) + g(x) + h(y)) \notag\\
\leqslant &\ d(f,h) + d(h,g).\notag
\end{align}

As this inequality holds for all $x,y \in X$, the triangle inequality
follows.
This proves (1), and as a consequence, $F_b(X,M) / \sim$ is always a
metric space under the metric~\eqref{Esemigroup}. That this metric is
translation-invariant~\eqref{Esemimetric} is straightforward.\smallskip

We now claim that (2) holds, first for the category ${\tt Semi}$. Indeed,
one defines the (bounded) function $f+g$ pointwise for $f,g \in
F_b(X,M)$. We claim that if $f \sim f'$ and $g \sim g'$ (i.e., they
have distances zero between them) in $F_b(X,M)$, then $f+g \sim f'+g'$.
This follows because
\begin{align*}
&\ d_M((f+g)(x) + (f'+g')(y), (f+g)(y) + (f'+g')(x))\\
= &\ d_M([f(x)+f'(y)] + [g(x)+g'(y)], [f(y)+f'(x)] + [g(y)+g'(x)]) = 0,
\end{align*}
which in turn follows from the equalities: $f(x)+f'(y) = f(y)+f'(x)$ and
$g(x)+g'(y) = g(y)+g'(x)$, for all $x,y \in X$.

Thus, $+$ is well-defined on $F_b(X,M) / \sim$. Similarly, one verifies
that if $f_n \to f$ and $g_n \to g$ in $F_b(X,M) / \sim$, then $f_n + g_n
\to f+g$. Next, given a semigroup morphism $\varphi : M \to N$,
post-composing by $\varphi$ defines a map $\varphi \circ -$ of semigroups
$: F_b(X,M) \to F_b(X,N)$; and if $d(f,g) = 0$ in $F_b(X,M)$, then
$d(\varphi \circ f, \varphi \circ g) = 0$ in $F_b(X,N)$. Thus $\varphi$
induces a well-defined map
\begin{equation}\label{Emorphism}
[\varphi] : F_b(X,M) / \sim \ \ \to F_b(X,N) / \sim.
\end{equation}

Finally, we verify that the given functor induces isometries on ${\rm
Hom}$-spaces. The first sub-step is to claim that every ${\rm Hom}$-space
${\rm Hom}(M,N)$ in ${\tt Semi}$ is itself a semigroup, with
translation-invariant metric given by:
\[
d(\eta,\varphi) \coloneqq \sup_{m \neq m' \in M} \frac{d_N(\eta(m) +
\varphi(m'), \varphi(m) + \eta(m'))}{d_M(m,m')}.
\]
We only show that if $d(\eta,\varphi) = 0$ then $\eta = \varphi$; the
remainder of the claim is straightforward (with the triangle inequality
following similarly to \eqref{Etriangle}). Indeed, if $d(\eta,\varphi) =
0$, then:
\[
d(\eta,\varphi) = 0 \quad \implies \quad d_M(\eta(m) +  \varphi(2m), 
\eta(2m) + \varphi(m)) = 0\ \forall m \in M \quad \implies \quad \eta
\equiv \varphi.
\]

\noindent This proves the above claim.

We now show that the map $\varphi \mapsto [\varphi]$ is an isometry
\[
[-] : {\rm Hom}_{\mathscr{C}}(M,N) \to
{\rm Hom}_{\mathscr{C}}(F_b(X,M) / \sim, F_b(X,N) / \sim).
\]
Suppose $\eta, \varphi : M \to N$ are semigroup morphisms, and let
$d(\eta, \varphi) = L$. If $L=0$ then the preceding computation shows
$\eta \equiv \varphi$ and hence $[\eta] = [\varphi]$. Otherwise, we
compute from first principles:
\begin{align*}
d([\eta], [\varphi]) = &\ \sup_{[f] \neq [g] \in F_b(X,M) / \sim}
\frac{d(\eta \circ [f] + \varphi \circ [g], \varphi \circ [f] + \eta
\circ [g])}{d([f],[g])}\\
= &\ \sup_{[f] \neq [g] \in F_b(X,M) / \sim} \frac{1}{d([f],[g])}
\sup_{x,x' \in X} d_N(\eta(m) + \varphi(m'), \varphi(m) + \eta(m')),
\end{align*}
where $m \coloneqq f(x) + g(x')$ and $m' \coloneqq f(x') + g(x)$.
Now note that
\[
d_N(\eta(m) + \varphi(m'), \varphi(m) + \eta(m')) \leqslant d(\eta,
\varphi) d_M(f(x) + g(x'), f(x') + g(x)) \leqslant L \cdot d([f],[g])
\]
for all $x,x' \in X$. It follows that $d([\eta], [\varphi]) \leqslant L =
d(\eta,\varphi)$.

To show the reverse inequality, suppose $m_l \neq m'_l, l \in \mathbb{N}$
are sequences in $M$ such that the sequences
\[
d_l \coloneqq \frac{d_N(\eta(m_l) + \varphi(m'_l), \varphi(m_l) +
\eta(m'_l))}{d_M(m_l,m'_l)}
\]
are non-decreasing to $L = d(\eta,\varphi)$ as $l \to \infty$. We now use
that $X$ is not a singleton, whence for a fixed element $x_1 \in X$, we
consider
\[
f_l|_X \equiv m_l, \qquad g_l|_{X \setminus x_1} \equiv m_l, \qquad
g_l(x_1) \coloneqq m'_l, \qquad l \in \mathbb{N}.
\]
Clearly $d([f_l], [g_l]) = d_M(m_l,m'_l)$, whence for any $x_2 \in X
\setminus x_1$, we compute from above:
\begin{align*}
&\ d([\eta], [\varphi])\\
\geqslant &\ \sup_{l \in \mathbb{N}} \frac{d_N(\eta(f_l(x_1) + g_l(x_2))
+ \varphi(f_l(x_2) + g_l(x_1)), \eta(f_l(x_2) + g_l(x_1)) +
\varphi(f_l(x_1) + g_l(x_2)))}{d([f_l],[g_l])}\\
= &\ \sup_{l \in \mathbb{N}} \frac{d_N(\eta(m_l) + \varphi(m'_l),
\varphi(m_l) + \eta(m'_l))}{d_M(m_l, m'_l)}\\
= &\ \sup_{l \in \mathbb{N}} d_l = L = d(\eta,\varphi).
\end{align*}

This proves the theorem for metric semigroups, i.e.~for $\mathscr{C} = {\tt
Semi}$.
We next impose the additional structure in each smaller subcategory one
by one, and show the result for the remaining $\mathscr{C}$. Clearly, if
$M$ is a monoid, then so is $F_b(X,M) / \sim$, with identity $f \equiv
0_M$. Now the result is easily verified for $\mathscr{C} = {\tt Mon}$.
(Note that all morphisms are automatically monoid maps.)

Next suppose $\mathscr{C} = R$-${\tt Mod}$. One checks that if $f \sim g$
then $rf \sim rg$, where $rf \in F_b(X,M)$ is defined in the usual
(pointwise) fashion. Also, multiplication by $r$ is Lipschitz on
$F_b(X,M) / \sim$ if it is so on $M$ itself. Now the result is easily
verified in this setting. Finally, if $M$ is also $R$-normed, then one
checks that so is $F_b(X,M) / \sim$.

This shows the result for all categories except for
$\overline{R}$-${\tt Mod}, \overline{R}$-${\tt NMod}$.
For these latter cases, it suffices to show the claim that if $M$ is a
complete abelian metric group then so is $F_b(X,M) / \sim$. We begin by
isolating the main component of this argument into a standalone result
(together with some related preliminaries).

\begin{lemma}\label{Lmetric}
Fix a set $X$ and an abelian metric semigroup $(M,+,d_M)$.
\begin{enumerate}
\item $F_b(X,M)$ is a metric space under the sup-norm
\[
d_\infty(f,g) \coloneqq \sup_{x \in X} d_M(f(x), g(x)).
\]
Moreover, $F_b(X,M)$ is complete if and only if $M$ is complete.

\item The quotient map of metric spaces $: F_b(X,M) \to F_b(X, M) / \sim$
is Lipschitz of norm at most $2$.

\item If $M$ is moreover a group, there exists a section $\Phi_{x_0} :
F_b(X,M) / \sim \ \ \to F_b(X,M)$ which is a sub-contraction:
\begin{equation}\label{Esupdist}
d_\infty( \Phi_{x_0}([f]), \Phi_{x_0}([g]) ) \leqslant d([f], [g]),
\qquad \forall x_0 \in X, \ [f], [g] \in F_b(X,M),
\end{equation}
and such that the image of $\Phi_{x_0}$ is precisely the set of functions
vanishing at $x_0$.
\end{enumerate}
\end{lemma}

\begin{proof}
(1) is well-known, 
\edittwo{Certainly if $F_b(X,M)$ is complete then restricting to the
constant functions, it follows that so is $M$. Conversely, given a
uniformly Cauchy sequence $f_n : X \to M$ with $M$ complete, define $f :
X \to M$ at each $x \in X$ to be the pointwise limit of the Cauchy
sequence $\{ f_n(x) : n \geqslant 1 \}$. We claim that $d_\infty(f_n, f)
\to 0$. Indeed, given  $\epsilon > 0$, there exists $N_0 \gg 0$ such that
$d_\infty(f_n, f_m) < \epsilon/2$ for $m,n > N_0$. Now fix $n > N_0$, and
for each $x \in X$, let $m_x \gg N_0$ be such that $d_M(f_{m_x}(x), f(x))
< \epsilon/2$. By the triangle inequality, $d_M(f_n(x), f(x)) <
\epsilon$. As this holds for every $n > N_0$ and $x \in X$, we have
$d_\infty(f_n, f) \to 0$.}
and (2) is standard using:
\begin{align*}
&\ d_M(f(x) + g(x'), f(x') + g(x))\\
\leqslant &\ d_M(f(x) + g(x'), g(x) + g(x')) +
d_M(g(x) + g(x'), g(x) + f(x'))\\
\leqslant &\ 2 d_\infty(f,g), \qquad \forall
f,g \in F_b(X,M), \ x,x' \in X.
\end{align*}
To show (3), choose any representative $f$ of $[f] \in F_b(X,M) / \sim$,
recalling by Remark~\ref{Rdiameter} that $f$ is unique up to translation
by an element of $M$. Now the `Kuratowski' map $\Phi_{x_0}([f]) \coloneqq
f(x) - f(x_0)$ satisfies \eqref{Esupdist}. (We also point out for
completeness some related observations at the start of \cite[Section
3]{CS}.)
\end{proof}

Returning to the proof of the above claim, suppose $[f_n] \in F_b(X,M) /
\sim$ is Cauchy, with $M$ complete. For any fixed $x_0 \in X$, this
implies by Lemma \ref{Lmetric}(3) that $\Phi_{x_0}([f_n])$ is Cauchy,
whence it converges in the $d_\infty$ metric to some bounded map $f$ by
Lemma \ref{Lmetric}(1). Hence, $[f_n] \to [f]$ by Lemma \ref{Lmetric}(2).
\end{proof}

We now refine the above categorical construction, when the domain $X$ is
additionally equipped with a topology:

\begin{corollary}\label{Cfunctor}
Given a topological space $X$ and an abelian metric semigroup $(M, +,
d_M)$, define $C_b(X,M)$ to be the set of bounded continuous functions $:
X \to M$. Now fix a unital subring $R \subset \R$.
\begin{enumerate}
\item If $M$ is in fact an abelian group, then $C_b(X,M) / \sim$ is a
closed subobject of $F_b(X,M) / \sim$.

\item With notation as in Theorem \ref{Tsemigroup} (for any of the
categories $\mathscr{C}$), $C_b(X,M)$ forms a pseudometric subspace of
$F_b(X,M)$, whence $C_b(X,M) / \sim$ is a subobject of $F_b(X,M) / \sim$
in $\mathscr{C}$. Moreover, $M \mapsto C_b(X,M) / \sim$ is also a
covariant endofunctor of $\mathscr{C}$, which is isometric if $X$ is not
a singleton.
\end{enumerate}
\end{corollary}

Note that the case of $X$ compact Hausdorff and $M = \R$ was studied in
greater detail in \cite{CS}, and is part of a broader program to study
isometries and linear isomorphisms of spaces of continuous functions. See
e.g.~\cite{Ar,CS,JV}, and the references therein.

\begin{proof}
To show (1), if $f_n : X \to M$ are continuous and $[f_n] \to [f]$ for
some $f \in F_b(X,M)$, then $\Phi_{x_0}([f_n]) \to \Phi_{x_0}([f])$
uniformly by~\eqref{Esupdist}, whence $\Phi_{x_0}([f])$ is continuous 
\edittwo{by the uniform convergence theorem}
and hence so is $f$.

For the categories whose objects are not all complete, the assertion (2)
follows from Theorem \ref{Tsemigroup} and the continuity of the
$R$-module operations. For the categories $\mathscr{C} =
\overline{R}$-${\tt Mod}, \overline{R}$-${\tt NMod}$, one further uses
the previous part and that $F_b(X,M) / \sim$ is complete by Theorem
\ref{Tsemigroup}.
\end{proof}

\section{Distances between $p$-norms; proof of the main result}\label{S3}

In this section we prove Theorem \ref{Tmain}, deriving part of it from
the functorial framework discussed in the previous section.

\subsection{$p$-norm computations}

Part of the proof of Theorem~\ref{Tmain} works with the $p$-norms on
$\F^k$; thus, we begin by providing `more standard' models for certain
sets of such norms. Given $p \in [1,\infty)$, define for $\bx = (x_1,
\dots, x_k) \in \F^k$ its $p$-norm:
\begin{equation}
\| (x_1, \dots, x_k) \|_p \coloneqq \left( |x_1|^p + \cdots + |x_k|^p
\right)^{1/p},
\end{equation}

\noindent and also define $\| (x_1, \dots, x_k) \|_\infty \coloneqq
\max_j |x_j|$.

As is well-known in the Banach--Mazur framework for $\F = \R$ or
$\C$~\cite{GKM}, if $p,q \in [1,\infty]$ and $2-p, 2-q$ have the same
sign, then the norms $\| \cdot \|_p$ and $\| \cdot \|_q$ on $\F^k$ have
Banach--Mazur distance $|1/p - 1/q| \cdot \log(k)$. However, this does
not usually hold for $1 \leqslant p < 2 < q \leqslant \infty$ -- for
instance if $k=2$ then $\| \cdot \|_1$ and $\| \cdot \|_\infty$ denote
the same point in the Banach--Mazur compactum. 
In the present setting of $\mathcal{S}'_k(\F)$, the $p$-norms share the
above behavior for $p \in [1,2]$ and $[2,\infty]$, but differ in the
metric structure for $[1,\infty]$:

\begin{proposition}\label{Plp}
Let $\mathcal{S}'_k(\F) \subset \mathcal{S}_k(\F)$ denote the equivalence
classes of the norms $\{ \| \cdot \|_p : 1 \leqslant p \leqslant \infty
\}$. Then the map $f : \mathcal{S}'_k(\F) \to [0,\log k]$, given by $f(\|
\cdot \|_p) \coloneqq \frac{\log k}{p}$ for $p \in [1,\infty)$ and $f(\|
\cdot \|_\infty) \coloneqq 0$, is an isometric bijection.
\end{proposition}

Thus the $p$-norms behave `uniformly well': $\mathcal{S}'_k(\R) \cong
\mathcal{S}'_k(\C) \cong \mathcal{S}'_k(\HH) \cong [0,\log k]$.

\begin{proof}
For $1 \leqslant p < q < \infty$, H\"older's inequality implies
$k^{-1/p} \| \bx \|_p \leqslant k^{-1/q} \| \bx \|_q$ for all $\bx \in
\F^k$, and equality is attained at the vectors with all equal
coordinates. For the other way, we claim that $\| \bx \|_p \geqslant \|
\bx \|_q$, with equality along the coordinate axes. Indeed, by rescaling
one may assume $\| \bx \|_p = 1$, whence $|x_j|^p \leqslant 1$ for all
$j$. Thus $|x_j| \leqslant 1$, and it follows that
\[
\| \bx \|_q^q = \sum_j |x_j|^q \leqslant \sum_j |x_j|^p = 1,
\]
whence $\| \bx \|_q \leqslant 1 = \| \bx \|_p$. From this it follows that
$d_{\mathcal{S}_k(\F)}(\| \cdot \|_p, \| \cdot \|_q) = \log k^{1/p -
1/q}$.

Finally, it is evident that $\| \bx \|_\infty \leqslant \| \bx \|_p
\leqslant k^{1/p} \| \bx \|_\infty$ for all $p \in [1,\infty)$ and $\bx
\in \F^k$, with equality attained in the same two cases as above. Thus,
$d_{\mathcal{S}_k(\F)}(\| \cdot \|_p, \| \cdot \|_\infty) = \log
k^{1/p}$. This concludes the proof.
\end{proof}

\begin{exam}
As another example, notice using Lemma \ref{Lbasic} that given a
non-negative measure $\mu$ supported on $[1,\infty)$, the function
\[
N_\mu(\bx) \coloneqq \int_1^\infty \| \bx \|_p\ d \mu(p)
\]
is a norm, if convergent on $\F^k$. Now the same reasoning as in the
above proof shows that for all such $\mu \geqslant 0$ with bounded
support and positive mass,
\begin{equation}
d_{\mathcal{S}_k(\F)}(N_\mu, \| \cdot \|_q) = \log \frac{\int_1^q
k^{1/p}\ d \mu(p)}{k^{1/q} \int_1^q d \mu(p)} \leqslant \log(k)(1 -
q^{-1}),
\end{equation}

\noindent where $\sup ({\rm supp}\ \mu) < q < \infty$.
\end{exam}

The above provide examples of subsets of $\mathcal{S}_k(\F)$ with bounded
diameter. However, this does not always happen, and we now mention such
an example, which also serves to show the `unboundedness' assertion in
the main result.

\begin{exam}
Given $p \in [1,\infty]$, $q \in [0,\infty)$, and an integer $1 \leqslant
j \leqslant k$, define
\begin{equation}
N_{p,q,j}({\bf x}) := \| {\bf x} \|_p + q |x_j|, \qquad {\bf x} \in \F^k
\end{equation}
and consider the family of such norms for a fixed $p$:
\begin{equation}\label{Eaxes}
\mathcal{S}_{k,p} := \{ N_{p,q,j} : q \in [0,\infty), \ j \in [k] \}.
\end{equation}
(one verifies easily that these are norms).
We now claim that akin to Proposition~\ref{Plp} for the $p$-norms, the
family $\mathcal{S}_{k,p}$ can also be realized as a more familiar metric
subspace of a Banach space:

\begin{proposition}\label{Pskp}
Suppose $k>1$. The subset $\mathcal{S}_{k,p}$ defined in~\eqref{Eaxes}
isometrically embeds into $\R^k$ with the $\ell^1$-norm, via $N_{p,q,j}
\mapsto \log (1+q) {\bf e}_j$. The image of $\mathcal{S}_{k,p}$ is the
union of the non-negative coordinate semi-axes.
\end{proposition}

\begin{proof}
Notice that $N_{p,q,j}({\bf x}) \leqslant (1+q) N_{p,q',j'}({\bf x})$ for
all $q,q' \geqslant 0$ and $j \neq j' \in \{ 1, \dots, k \}$, with
equality attained at least for ${\bf x} = x {\bf e}_j$ (here, ${\bf e}_1,
\dots, {\bf e}_n$ comprise the standard basis of $\F^k$). It follows that
\[
d_{\mathcal{S}_k(\F)}(N_{p,q,j}, N_{p,q',j'}) = \log(1+q) + \log(1+q').
\]
One next shows that for a fixed $j \in \{ 1, \dots, k \}$,
\[
0 \leqslant q \leqslant q' < \infty \quad \implies \quad
N_{p,q,j}(\bx) \leqslant N_{p,q',j}(\bx) \leqslant \frac{1+q'}{1+q}
N_{p,q,j}(\bx), \qquad \forall \bx \in \F^k,
\]
with equality possible on $\F^k \setminus \{ {\bf 0} \}$ in either
inequality (note that equality in the lower bound requires $k>1$).
Therefore $d_{\mathcal{S}_k(\F)}(N_{p,q,j}, N_{p,q',j}) = \log(1+q') -
\log(1+q)$. This concludes the proof.
\end{proof}
\end{exam}

\subsection{Proof of the main result}

With Proposition~\ref{Pskp} and the functorial analysis in the previous
section in hand, we can show our main result.

\begin{proof}[Proof of Theorem \ref{Tmain}]
Begin by fixing any compact subset
\begin{equation}\label{Echoose}
X \subset \F^k \setminus \{ {\bf 0} \} \cong \R^{dk} \setminus \{ {\bf 0}
\} \qquad \text{satisfying:} \qquad
\forall \bx \in \F^k \setminus \{ {\bf 0} \}, \ \exists
\alpha_\bx \in \F^\times \text{ such that } \alpha_\bx \bx \in X.
\end{equation}

\noindent (For instance, $X$ could be the unit sphere $S^{dk-1}$.)
The bulk of the proof involves showing the claim that the space
$\mathcal{S}_k(\F)$ is a closed metric subspace of the Banach space
$C(X, \R) / \sim \; \ = C_b(X, \R) / \sim$ (see Corollary \ref{Cfunctor},
noting that $X$ is compact). In particular, $\mathcal{S}_k(\F)$ is
complete.

To show the claim, we construct the embedding $\Psi : \norm(\F^k) \to
C(X, \R)$ as follows: given a norm $N \in \norm(\F^k)$, define $\Psi(N)
\coloneqq \log N|_X \in C(X,\R)$. Since $N$ is a norm, it is uniquely
determined by its restriction to $X$, whence $\Psi$ is injective.
Moreover, the respective notions of $\sim$ are compatible via taking the
logarithm, whence $\Psi$ induces an injection $[\Psi] : \mathcal{S}_k(\F)
\hookrightarrow C(X,\R) / \sim$ of metric spaces. It is easily verified
that $[\Psi]$ is an isometry; recall here that the metric on
$\mathcal{S}_k(\F)$ is given by: $d([N], [N']) \coloneqq
\log(M_{N,N'}/m_{N,N'})$ (see Equations~\eqref{Emetric}
and~\eqref{Elip}).

It remains to show closedness. Suppose $N_l$ are norms on $\F^k$ such
that $[\log N_l|_X] \to [f]$ in $C(X,\R) / \sim$ under the metric
in~\eqref{Emetric} (recall Corollary~\ref{Cfunctor} here). As above, one
can choose representative norms $N_l$ on $\F^k$ and a function $f \in
C(X,\R)$ in their equivalence classes, such that for all $l>0$,
\begin{equation}\label{Econv}
(\log N_l) - f : X \to [0,\epsilon_l], \quad \epsilon_l \geqslant 0,
\end{equation}

\noindent with $\epsilon_l \to 0^+$ as $l \to \infty$. In particular,
$N_l \to \exp(f)$ pointwise on $X$. Define
\[
N({\bf 0}) \coloneqq 0, \qquad N(\bx) \coloneqq |\alpha_\bx|^{-1}
\exp(f(\alpha_\bx \bx)), \ \forall \bx \in \F^k \setminus \{ {\bf 0} \},
\]
where $\alpha_\bx$ comes from the defining property of $X$. The above
claim is proved if one shows that $N$ is a norm on $\F^k$. First note
that $N$ is indeed well-defined: if $\alpha_\bx$ and $\beta_\bx$ are such
that $\alpha_\bx \bx, \beta_\bx \bx \in X$ for some $\bx$, then
\[
|\alpha_\bx|^{-1} \exp( f( \alpha_\bx \bx) ) =
\lim_{l \to \infty} |\alpha_\bx|^{-1} N_l(\alpha_\bx \bx)
= \lim_{l \to \infty} |\beta_\bx|^{-1} N_l(\beta_\bx \bx)
= |\beta_\bx|^{-1} \exp( f( \beta_\bx \bx) ).
\]

Next, $N$ is homogeneous: given any scalar $\beta \in \F$ and vector $\bx
\in \F^k \setminus \{ {\bf 0} \}$,
\[
\frac{N(\beta \bx)}{N(\bx)} = \frac{|\alpha_{\beta \bx}|^{-1} \exp
f(\alpha_{\beta \bx} \beta \bx)}{|\alpha_\bx|^{-1} \exp f(\alpha_\bx
\bx)} = \lim_{l \to \infty} \frac{|\alpha_{\beta \bx}|^{-1}
N_l(\alpha_{\beta \bx} \beta \bx)}{|\alpha_\bx|^{-1} N_l(\alpha_\bx \bx)}
= \lim_{l \to \infty} |\beta| = |\beta|.
\]

Finally, observe that $N$ is sub-additive, i.e., $N(\bx + \by) \leqslant
N(\bx) + N(\by)$ for $\bx, \by \in \F^k$. Indeed, this is immediate if
any of $\bx, \by, \bx + \by$ is zero, so assume this does not happen, and
compute:
\begin{align*}
N(\bx + \by) = &\ |\alpha_{\bx + \by}|^{-1} \lim_{l \to \infty}
N_l(\alpha_{\bx + \by} (\bx + \by))\\
\leqslant &\ |\alpha_{\bx + \by}|^{-1} \lim_{l \to \infty}
N_l(\alpha_{\bx + \by} \bx) + N_l(\alpha_{\bx + \by} \by)\\
= &\ |\alpha_{\bx + \by}|^{-1} \lim_{l \to \infty} \left( 
\frac{| \alpha_{\bx + \by}|}{|\alpha_\bx|} N_l(\alpha_\bx \bx) +
\frac{| \alpha_{\bx + \by}|}{|\alpha_\by|} N_l(\alpha_\by \by) \right)\\
= &\ N(\bx) + N(\by).
\end{align*}
The closedness of $\mathcal{S}_k(\F)$ now follows, whence by
Theorem~\ref{Tsemigroup} and Corollaries~\ref{C18} and~\ref{Cfunctor}, we
have a chain of inclusions with closed images
\[
\mathcal{S}_k(\F) \hookrightarrow \mathcal{S}_{dk}(\R) \hookrightarrow
C_b(X,\R) / \sim \ \hookrightarrow F_b(X, \R) / \sim \; ;
\]
The above claim now follows; hence $\mathcal{S}_k(\F)$ is complete. Next,
Lemma \ref{Lbasic} implies that $\norm(\F^k)$ is convex, hence
path-connected, whence so is $\mathcal{S}_k(\F)$.
Moreover, clearly $\mathcal{S}_1(\F)$ is a point.
Finally, assuming $k>1$, Proposition~\ref{Pskp} shows that
$\mathcal{S}_k(\F)$ is unbounded.
\end{proof}

The above proof shows that the metric space of norms embeds into $C(X,\R)
/ \sim$ for many different compact topological subspaces $X \subset \F^k$
(see \eqref{Echoose}). We conclude by exploring these embeddings in
greater detail, fixing $\F = \R$ for convenience.
Specifically, if $X = S^{k-1}$ denotes the unit sphere, then under the
embedding $\mathcal{S}_k(\R) \hookrightarrow \mathbb{B} \coloneqq
C(S^{k-1},\R) / \sim$, the origin in the Banach space $\mathbb{B}$ is
precisely the image of the $2$-norm $\| \cdot \|_2$. More generally, we
have such an identification of the origin for \textit{every} norm -- but
not for every space $X$.

\begin{proposition}\label{Porigin}
Given a norm $N : \R^k \to \R$ and any radius $r>0$, let $X_{N,r}$ denote
the sphere
\[
X_{N,r} \coloneqq \{ \bx \in \R^k : N(\bx) = r \}.
\]
Now fix a norm $N$ on $\R^k$, as well as a subset $X \subset \R^k$ such
that every nonzero vector is a positive real multiple of a point in $X$.
Then the following are equivalent:
\begin{enumerate}
\item $\mathcal{S}_k(\R)$ embeds as a closed subset in the Banach space
$C(X,\R) / \sim$ for some topological space $X$ via $[N'] \mapsto [\log
N'|_X]$; and this embedding maps the equivalence class $[N]$ to the
origin.

\item $X = X_{N,r}$ for some $r>0$.
\end{enumerate}

\noindent However, there exist compact sets $X \subset \R^k$ such that
the image of the embedding $\mathcal{S}_k(\R) \hookrightarrow C(X,\R) /
\sim$ avoids the origin.
\end{proposition}

\begin{proof}
If (1) holds, then $N|_X$ must be constant, whence $X \subset X_{N,r}$
for some $r>0$. Moreover, if $\bx \in X_{N,r}$, then $\alpha \bx \in X$
for some $\alpha>0$. Since $N(\alpha \bx) = \alpha r$, it follows that
$\alpha = 1$ and hence $X = X_{N,r}$.

Conversely, suppose (2) holds. Since all norms on $\R^k$ are equivalent,
the space $X_{N,r}$ is compact and hence satisfies \eqref{Echoose},
whence the proof of Theorem \ref{Tmain} applies to it. In particular, the
image of $N$ under the embedding $: \mathcal{S}_k(\R) \hookrightarrow
C(X_{N,r}, \R) / \sim$ is a constant function, whose image under $\sim$
is the trivial class.

Finally, given any point $\by \neq 0$, define the sphere
\[
X_\by \coloneqq \{ \bx \in \R^k : \| \bx - \by \|_2^2 = 1 + \| \by
\|_2^2 \}.
\]
It is clear that ${\bf 0}$ is in the `interior' of the sphere. Also
notice that for every unit direction ${\bf v} \in S^{k-1}$, there exists
a unique $\alpha>0$ such that $\alpha {\bf v} \in X_\by$. Indeed, from
the conditions
\[
\| \alpha {\bf v} - \by \|^2_2 = 1 + \| \by \|^2_2, \qquad \alpha > 0,
\]
one derives: $\alpha = \sqrt{1 + \langle {\bf v}, \by \rangle^2} +
\langle {\bf v}, \by \rangle$. In particular, $X_\by$ satisfies
\eqref{Echoose} and hence the proof of Theorem \ref{Tmain} applies to it.
However, no norm maps via the embedding $\mathcal{S}_k(\R)
\hookrightarrow \mathbb{B} \coloneqq C(X_\by,\R) / \sim$ to the origin in
the Banach space $\mathbb{B}$. Indeed, if $[N] \mapsto 0_{\mathbb{B}}$,
then the norm $N$ would restrict to a constant on $X_\by$. But this is
false: $X_\by$ intersects the line $\R \by$ at the two points $( \| \by
\|_2 \pm \sqrt{1 + \| \by \|_2}) \by$, and as these are not negatives of
one another, evaluating $N$ yields unequal values.
\end{proof}

\section{The normed space of metrics on a finite set}\label{Sdistortion}

We now study a parallel setting to the metric space of norms on $\F^k$,
in which the above functorial approach is also applicable. Given a finite
set $[n] := \{ 1, \dots, n \}$ with $n \geqslant 2$, it is possible to
impose a pseudometric on the space of metrics on $[n]$ in the same way as
above: given metrics $\rho, \rho' : X \times X \to [0,\infty)$, define
\[
d_{[n]}(\rho, \rho') :=
\log \ \max_{j \neq k} \frac{\rho'(x_j, x_k)}{\rho(x_j, x_k)}
\ -\ \log \ \min_{j \neq k} \frac{\rho'(x_j, x_k)}{\rho(x_j, x_k)};
\]
note that $\exp \circ d_{[n]}$ is termed the \textit{distortion} in
metric geometry and computer science.

We cite the well-known surveys~\cite{Li,Ma} for further details and
reading on the numerous applications of distortion and metric geometry to
computer science, combinatorics, and other fields. 
Also note that there is a different, well-studied pseudometric on the
space of metrics on $[n]$, or more generally on all compact metric
spaces. This is the Gromov--Hausdorff metric (which is not comparable to
$d_{[n]}$, as we see below). Nevertheless, the metric $d_{[n]}$, as well
as the connection between diameter norms and (log-)distortion, do not
seem to be studied or recorded in the literature. This motivates the
present section.

We begin with a result that is parallel to Theorem~\ref{Tmain} for
$\mathcal{S}_k(\F)$, and again follows from the above functorial
analysis. In particular, it shows that the metric $d_{[n]}$ is also a
diameter norm:

\begin{theorem}\label{Tmetric}
Fix an integer $n \geqslant 2$.
\begin{enumerate}
\item The map $d_{[n]}$ is a pseudometric on the space of metrics on
$[n]$, with equivalence classes precisely consisting of proportional
metrics.

\item The quotient metric space $\mathcal{S}([n])$ is a complete,
path-connected, metric subspace of the Banach space $\R^{\binom{[n]}{2}}
/ \sim \ \; = C(\binom{[n]}{2}, \R) / \sim$ with the diameter norm, where
$\binom{[n]}{2}$ denotes the discrete set of two-element subsets of
$[n]$.

\item The space $\mathcal{S}([n])$ is a singleton if $n=2$, and unbounded
otherwise.
\end{enumerate}
\end{theorem}

\begin{proof}
We only point out why for $n>2$ the space $\mathcal{S}([n])$ is
unbounded. Indeed, for each $m \geqslant 1$ let $X_m := \{ 1, 2, \dots,
n-1; n+m+1 \}$ be the (induced) metric subspace of $(\R, | \cdot |)$, and
let $\rho_{n,m} : [n] \times [n] \to \R$ be the metric induced by the
unique rank/order-preserving map $: X_m \to [n] \subset \R$. Compare
$\rho_{n,m}$ to the discrete metric $\rho(x,y) = 1 - \delta_{x,y}$: the
log-distortion between them is $\log (n+m)$, and $m$ can grow without
bound.
\end{proof}

\begin{remark}
The metric on $\mathcal{S}([n])$ -- henceforth denoted by
$d_{\mathcal{S}([n])}$ -- is not comparable to the well-studied
Gromov--Hausdorff metric on compact metric spaces:
\[
d_{GH}(X_1, X_2) := \inf_{Z, \iota_1, \iota_2} d_H(\iota_1(X_1),
\iota_2(X_2)),
\]
where one runs over all metric spaces $Z$ and isometric embeddings $:
\iota_j : X_j \hookrightarrow Z$; and where $d_H$ denotes the Hausdorff
distance in $Z$. To see why $d_{GH}$ is not comparable to the above
metric $d_{\mathcal{S}([n])}$, for any $n>2$ choose two different pairs
of points from $[n]$, say $\{ a, b \} \neq \{ c, d \} \subset [n]$. Now
define metrics $\rho, \rho'$ on $[n]$ via: $\rho(a,b) = \rho'(c,d) = 1/2$
and all other nonzero values of $\rho, \rho'$ are $1$. These two metric
spaces are clearly isometric under $a \leftrightarrow c, \ b
\leftrightarrow d$, and all other points left unchanged. However, the
metrics are not proportional. Going the other way, proportional but
unequal metrics on $[n]$ do not admit an isometry between them.
\end{remark}

Our next few results are meant to help better understand the metric space
$\mathcal{S}([n])$. The following result parallels
Proposition~\ref{Pskp}, and illustrates how a certain one-parameter
family of metrics on $[n]$ can be understood through a more standard
model:

\begin{proposition}\label{Ps1n}
Fix integers $n > 0$ and $1 \leqslant j \leqslant \binom{n}{2}$, as well
as any bijection to identify the set of pairs $\binom{[n]}{2}$
(i.e.~edges between points in $[n]$) with the set $\left[ \binom{n}{2}
\right] = \{ 1, \dots, \binom{n}{2} \}$. Given $a \in (0,1]$, let
$\rho_{j,a} : [n] \times [n] \to \R$ denote the metric in which all
nonzero distances in $[n]$ are $1$, except for the distance corresponding
to the edge $j$, which is $a$. Define
\[
\mathcal{S}'([n]) := \{ \rho_{j,a} : j \in \textstyle{\left[ \binom{n}{2}
\right]}, \ a \in (0,1] \}.
\]
Then $\mathcal{S}'([n])$ (or its set of equivalence classes) embeds
isometrically into $\R^{\binom{n}{2}}$ with the $\ell^1$-norm, via:
$\rho_{j,a} \mapsto (\log a) {\bf e}_j$. The image is the union of the
non-positive coordinate semi-axes.
\end{proposition}

The set $\mathcal{S}'([n])$ comprises the metrics in which $[n]$ may be
viewed as a weighted graph with all edge weights but one equal and at
least as large as the remaining edge weight.

\begin{proof}
Write $N := \binom{n}{2}$ for convenience. Viewing each metric $\rho$ on
$[n]$ as a function $: [N] \to (0,\infty)$, say $(\rho^{(1)}, \dots,
\rho^{(N)})^T$, it follows that
\[
d_{\mathcal{S}([n])}(\rho_{j,a}, \rho_{j',a'}) =
\log \max_{1 \leqslant k \leqslant N}
\frac{\rho_{j,a}^{(k)}}{\rho_{j',a'}^{(k)}} -
\log \min_{1 \leqslant k \leqslant N}
\frac{\rho_{j,a}^{(k)}}{\rho_{j',a'}^{(k)}}.
\]
Now if $j=j'$ then the distance is $|\log a - \log a'|$, else the
distance is $- \log a - \log a'$.
\end{proof}

\begin{remark}\label{Rs1n}
Notice that the function $\rho_{j,a}$ is a metric on $[n]$ if and only if
$a \in (0,2]$. Thus, one can compute the distance between $\rho_{j,a}$
and $\rho_{j',a'}$ for $a,a' \in (0,2]$. If $j=j'$ then we once again
obtain $|\log a - \log a'|$, but if $j \neq j'$ then we have:
\[
d_{\mathcal{S}([n])}(\rho_{j,a}, \rho_{j',a'}) =
\begin{cases}
|\log a| + |\log a'|, & \text{if either } a,a' \geqslant 1 \text{ or
} a,a' \leqslant 1;\\
\max \{ |\log a|, |\log a'| \}, \qquad & \text{otherwise}.
\end{cases}
\]
\end{remark}

We next study \textit{embeddings} in $\mathcal{S}([n])$ of metric spaces.
In doing so, we are motivated by recent work~\cite{IIT}, where it was
shown that every finite metric space of size at most $\binom{n}{2}$
embeds isometrically into the Gromov--Hausdorff space of isometry classes
of $n$-element metric spaces. In other words, a representative from the
Gromov--Hausdorff equivalence class of every metric space of size at most
$\binom{n}{2}$ embeds isometrically into Gromov--Hausdorff space. The
following result is parallel in spirit, for the space $\mathcal{S}([n])$:

\begin{theorem}\label{TSnembed}
Let $(X,d)$ be a finite metric space, and $n \geqslant 3$ an integer such
that $X \leqslant \binom{n}{2}$. Then there exists an equivalent metric
space to $(X,d)$ -- i.e., a rescaling of $d$ -- that admits an isometric
embedding into $\mathcal{S}([n])$.
\end{theorem}

\noindent Note that the result fails to hold for $n=2$, since
$\mathcal{S}([2])$ is a singleton.

Before proving Theorem~\ref{TSnembed}, we recall that its
Gromov--Hausdorff analogue in~\cite{IIT} was stated using the
\textit{smallest} $n$ such that $|X| \leqslant \binom{n}{2}$. While our
variant does not \textit{a priori} have this extra restriction, we point
out that the two versions are equivalent for $\mathcal{S}([n])$, because
of the following result:

\begin{proposition}\label{PSnembed}
For all $n \geqslant 2$, the metric space $\mathcal{S}([n])$
isometrically embeds into $\mathcal{S}([n+1])$.
\end{proposition}

\begin{proof}
Given a finite metric space $(X,d)$ with $|X|=n$, embed it into a metric
space $X \sqcup \{ n+1 \}$, where $n+1$ is an additional point
with distance $\diam(X)$ from every $x \in X$. A straightforward
computation (perhaps rescaling both diameters to $1$ for convenience) now
shows that this defines an isometry from $\mathcal{S}([n])$ into
$\mathcal{S}([n+1])$.
\end{proof}

We now prove the above theorem.

\begin{proof}[Proof of Theorem~\ref{TSnembed}]
In fact we will construct the embedding $: (X, \alpha \cdot d)
\hookrightarrow \mathcal{S}([n])$ for a specific $\alpha > 0$, using
several `natural' tools. We begin by describing these tools.

Observe that every metric on $[n]$ can be viewed as an element of
$(0,\infty)^N$ where $N = \binom{n}{2}$. Via taking logarithms, the space
$\mathcal{S}([n])$ is in bijection with the set $\Psi([n])$ of all tuples
$(\psi_{ij})^T \in \R^{\binom{[n]}{2}} \cong \R^N$ (here $i<j$) such that
\begin{equation}\label{Edistances}
\exp(\psi_{ij}) + \exp(\psi_{jk}) \geqslant \exp(\psi_{ik}), \qquad
\forall 1 \leqslant i < j < k \leqslant n.
\end{equation}
Note that rescaling the metric by $\alpha > 0$ is equivalent to
translating all $\psi_{ij}$ by $\log \alpha$. In other words, $\Psi([n])$
sits inside $\R^N / \sim$, where $\sim$ denotes additive translations by
scalar multiples of $(1,\dots,1)^T$.

The next observation is that for an integer $p>0$, the space $(\R^p, \|
\cdot \|_\infty)$ is isometrically isomorphic as a Banach space to
$\R^{p+1} / \sim$ with the diameter norm $\diam$, where $\sim$ denotes
quotienting by additive translation by multiples of $(1,\dots,1)^T$. More
generally, given integers $0 < p < q$, the map
\[
\Psi_{p,q} : (x_1, \dots, x_p)^T \mapsto (x_1, \dots, x_p, 0_{1 \times
(q-p)})^T + \R (1,\dots,1)_q^T
\]
is an isometric linear embedding of $(\R^p, \| \cdot \|_\infty)$ into
$(\R^q / \sim, \diam)$. For this reason, note in the previous paragraph
that the bijection of sets
\[
\log[-] : (\mathcal{S}([n]), d_{\mathcal{S}([n])}) \to (\Psi([n]), \diam)
\]
is in fact an isometry of metric spaces.

Finally, we recall the \textit{Fr\'echet embedding} \cite{Fr}, which maps
an $N$-element metric space $(X = \{ x_0, \dots, x_{N-1} \}, d)$
isometrically into $\R^{N-1}$ with the sup-norm, via: $x_j \mapsto
(d(x_1, x_j), \dots, d(x_{N-1}, x_j))^T$ for $0 \leqslant j \leqslant
N-1$. Let us denote this embedding by $Fr : X \to \R^{|X|-1}$.

With these ingredients in hand, we claim:

\begin{proposition}\label{Plog2}
Fix an integer $n \geqslant 3$ and a metric space $(X,d)$ such that $3
\leqslant |X| \leqslant N = \binom{n}{2}$. If $X$ has diameter at most
$\log 2$, then the composite map
\[
\Psi_{|X|-1,N} \circ Fr : (X,d) \hookrightarrow (\R^{|X|-1}, \| \cdot
\|_\infty) \hookrightarrow (\R^N / \sim, \diam)
\]
has image inside the metric space $\Psi([n]) \simeq \mathcal{S}([n])$.
The converse holds if $X$ is a three-element set.
\end{proposition}

Notice that Proposition~\ref{Plog2} implies Theorem~\ref{TSnembed}, by
rescaling the metric on $X$ by $(\log 2) / \diam(X)$. (The case of
$|X|=2$ is straightforward.)
\end{proof}

To complete the proof, it remains to show the preceding proposition.

\begin{proof}[Proof of Proposition~\ref{Plog2}]
If $\diam X \leqslant \log 2$, then we claim for each $x \in X$ that any
three coordinates of the Fr\'echet tuples $(d(x', x))_{x' \in X}$
satisfy~\eqref{Edistances}. Indeed, if $x_1, x_2, x_3 \in X$ then
\[
\exp d(x_3,x) \leqslant 2 \leqslant \exp d(x_1,x) + \exp d(x_2,x).
\]
This shows $\Psi_{|X|,N} \circ Fr (X) \subset \Psi([n])$. Conversely, let
$X = \{ x, y, z \}$; then we are assuming that
\[
\Psi_{|X|,N} \circ Fr(X) = \{ (d(x,y), d(x,z), 0, \dots, 0)^T, \ (0,
d(y,z), 0, \dots, 0)^T, \ (d(z,y), 0, 0, \dots, 0)^T \}
\]
is contained in $\Psi([n])$. Hence the last of the three points in $\R^N$
satisfies~\eqref{Edistances}, which in turn implies:
\[
\exp d(z,y) \leqslant 1+1 = 2.
\]
The same argument using the other two Fr\'echet embeddings of $X$ shows
that $\diam X \leqslant \log 2$.
\end{proof}

\section{Norms on arbitrary Banach spaces; concluding remarks and
questions}

\subsection{Parallel settings}

As shown in Section~\ref{S2}, diameter norms offer a unified and
functorial framework, which subsumes and explains both
Theorem~\ref{Tmain} about norms on $\F^k$, as well as
Theorem~\ref{Tmetric} about metrics on $[n]$. This treatment also applies
more generally, and we begin this final section by stating (without
proofs, and for completeness) two parallel results: in an arbitrary
Banach space and in a class of discrete metrics on an arbitrary set.

\begin{proposition}
Let $\mathbb{B}$ be an arbitrary Banach space over $\F = \R$ or $\C$, and
let $X \subset \mathbb{B} \setminus \{ {\bf 0} \}$ be a subset such that
for all ${\bf 0} \neq \bx \in \mathbb{B}$, there exists $\alpha_{\bx} \in
\F^\times$ such that $\alpha_{\bx} \bx \in X$. Then the set of
equivalence classes of norms on $\mathbb{B}$ (i.e., up to scaling by
$(0,\infty)$) whose restriction to $X$ is bounded away from $0,\infty$
can be isometrically realized as a complete, path-connected metric
subspace of the Banach space $C_b(X, \R) / \sim$ with the diameter norm.
\end{proposition}

Notice that if moreover $\mathbb{B}$ is finite-dimensional and $X$ is
compact then this reduces to Theorem~\ref{Tmain}. Similarly, for any set
$X$ we have the following extension of Theorem~\ref{Tmetric}:

\begin{proposition}
For any nonempty set $X$ of size at least $3$, the equivalence classes
(again by scaling) of metrics $d$ on $X$ bounded away from $0,\infty$
outside the diagonal -- i.e., such that
\[
0 < \inf_{x \neq x' \in X} d(x,x') \leqslant \sup_{x \neq x' \in X}
d(x,x') < \infty
\]
form a complete, path-connected unbounded metric subspace of the Banach
space $C_b(\binom{X}{2}, \R) / \sim$. Here $\binom{X}{2}$ denotes the
discrete set of pairs of elements in $X$.
\end{proposition}

\subsection{Further questions}

Following Theorems~\ref{Tmain} and~\ref{Tmetric} studying the norms on
$\F^k$ and the metrics on $[n]$ respectively, it may be interesting to
further explore the spaces $\mathcal{S}_k(\F)$ and $\mathcal{S}([n])$;
exploring the former may provide additional insights into the
Banach--Mazur compactum quotient space. Thus, we conclude with some
observations and questions in both of the above settings.
\begin{enumerate}
\item Are there more standard mathematical (geometric) models with which
one can identify the metric spaces $\mathcal{S}_k(\F)$ and
$\mathcal{S}([n])$? What can one say about their geometric
properties?\medskip

\item What are the automorphism groups of these spaces? (Depending on the
category under consideration, one may wish to study homeomorphisms,
isometries, \dots) For instance, by the final assertion in
Lemma~\ref{Lbasic}, $\mathcal{S}_k(\F)$ is equipped with the group
$PGL_{dk}(\R)$ of isometries, under a real-linear identification of
$\F^k$ with $\R^{dk}$.\footnote{It is easy to verify here that for $A \in
GL_{dk}(\R)$, $d_{\mathcal{S}_k(\F)}(N(A \cdot -), N'(A \cdot -)) =
d_{\mathcal{S}_k(\F)} (N(\cdot), N'(\cdot))$.} 
An additional observation (by Terence Tao in recent discussions) is that
$\mathcal{S}_k(\F)$ also carries an isometric involution, which arises
from considering dual norms.
A parallel observation is that the space $\mathcal{S}([n])$ is equipped
with an obvious symmetry group $S_n$ of automorphisms. (In contrast, the
Gromov--Hausdorff space has no isometries~\cite{IT2}.) One may also
consider local isometries of $\mathcal{S}([n])$ and $\mathcal{S}_k(\F)$,
as previously done for the Gromov--Hausdorff space in~\cite{IT1}.

\begin{remark}
For completeness we point out that individual norms can indeed be
unchanged under precomposing by elements of $PGL_{dk}(\R)$. For instance,
for the $\| \cdot \|_2$-norm one has the image of $O_{dk}(\R)$, while for
$p \in [1,\infty] \setminus \{ 2 \}$, results of Banach~\cite{Banach} and
Lamperti~\cite{Lamperti} show that `generalized permutation matrices' are
isometries of $\| \cdot \|_p$. These consist of the products of
permutation matrices with diagonal orthogonal or unitary matrices for $\F
= \R$ or $\C$ respectively. (When $\F = \R$, this is precisely the Weyl
group of type $B$ or $C$, i.e.~the hyperoctahedral group $S_2 \wr S_n$ of
signed permutations.)

At the same time, say for $\F = \R$ there is no nontrivial matrix $A \in
GL_k(\R), \ A \not\in \R^\times \cdot {\rm Id}$, whose precomposition
fixes all of $\mathcal{S}_k(\F)$. Indeed, using a $\| \cdot \|_p$-norm
for $p \neq 2$, by the previous paragraph $A$ must be a nonzero scalar
multiple of some signed permutation matrix $A' \in S_2 \wr S_n$, say $A =
cA'$. Suppose the nonzero entries of $A'$ correspond to the (signed)
permutation $\sigma \in S_n$. Now let $N({\bf x}) := \sum_{j=1}^k j
|x_j|$ for ${\bf x} \in \R^k$. If ${\bf e}_1, \dots, {\bf e}_k$ comprise
the standard basis elements of $\R^k$, and $N(A {\bf x}) \equiv c' N({\bf
x})$ on $\R^k$ for some $c' > 0$, then
\[
N(A {\bf e}_j) = c' N({\bf e}_j)\ \forall j \quad \implies \quad c' j =
|c| \sigma^{-1}(j), \ \forall j.
\]
Multiplying these inequalities yields: $|c| = c'$. Now evaluating at
${\bf e}_j + {\bf e}_{\sigma^{-1}(j)}$ yields:
\[
j + \sigma^{-1}(j) = \sigma^{-1}(j) + \sigma^{-2}(j), \ \forall j \in
[k].
\]
Hence $\sigma$ has order at most $2$. Using this and evaluating at ${\bf
e}_j + 2 {\bf e}_{\sigma^{-1}(j)}$ yields:
\[
j + 2 \sigma^{-1}(j) = \sigma^{-1}(j) + 2 j, \ \forall j
\]
and we conclude that $\sigma = {\rm Id}$. Finally, suppose two diagonal
entries of $A$ are unequal, say $a_{11} = c', a_{22} = -c'$. Define the
norm $N({\bf x}) := \| {\bf x} \|_1 + |x_1 + x_2|$, and evaluate it at
${\bf x} = (1,1,0,\dots,0)^T$:
\[
0 = c' N({\bf x}) - N(A {\bf x}) = 4c' - 2c'.
\]
Since $c' > 0$, our supposition must therefore be false, concluding the
proof. \qed
\end{remark}

\item How does the space $\mathcal{S}_k(\F)$ relate to
$\mathcal{S}_{k+1}(\F)$? Observe by Proposition~\ref{Plp} that the
$p$-norms isometrically map to the $p$-norms, provided one rescales the
metric/norm on each $\mathcal{S}_k(\F)$ by $\log(k)$. Alternately,
without rescaling any of the norms on $\mathcal{S}_k(\F)$, is it possible
to compute the fibers of `the' restriction map : $\mathcal{S}_{k+1}(\F)
\to \mathcal{S}_k(\F)$?

On a related note (say with $\F = \R$ for convenience), is this
restriction map a surjection? I.e., is there a ``Hahn--Banach'' extension
of every norm on $\R^k$ to one on $\R^{k+1}$, say minimally
increasing/without increasing the (log-)distortion relative to some
reference norm?\medskip

\item Notice that the previous question has a variant for
$\mathcal{S}([n])$ with a positive answer, by Proposition~\ref{PSnembed}.
Moreover, the fibers of the restriction of norms from $[n+1]$ to $[n]$
are solution sets to finite systems of inequalities. It may be
interesting to study the structures of these solution spaces.\medskip

\item To understand the `sizes' and growth of balls in these spaces, one
can also explore their metric entropy. Recall for a metric space $X$ and
a radius $r>0$, the metric entropy of $E \subset X$ is the largest number
of points in $E$ that are $r$-separated. This is related to the internal
and external covering numbers and the packing number of $E$; we refer the
reader to~\cite{Tao} for a detailed introduction to these ideas.\medskip

\item What is the smallest Banach space inside which these spaces
(or distinguished subsets therein) can be isometrically embedded? Of
course if we restrict to finite subsets $X$ then the classic observation
of Fr\'echet \cite{Fr} shows that $(X,d)$ isometrically embeds into
$\R^{|X|-1}$ with the supnorm (and into $\R^{|X|-2}$ if $|X| \geqslant
4$) -- see the discussion prior to Proposition~\ref{Plog2}.

If instead of the supnorm one is interested in Euclidean space embeddings
-- for subsets of $\mathcal{S}_k(\F)$ or for $\mathcal{S}([n])$ -- the
classic paper of Schoenberg \cite{Sch} (following related works in metric
geometry by Menger, Fr\'echet, von Neumann, and others) provides the
following result for finite metric spaces $X$:

\begin{theorem}[{Schoenberg~\cite{Sch}, 1935}]\label{Tmenger}
Fix integers $n,r \geqslant 1$, and a finite set $X = \{ x_0, \dots, x_n
\}$ together with a metric $d$ on $X$. Then $(X,d)$ isometrically embeds
into $\R^r$ (with the Euclidean distance/norm) but not into $\R^{r-1}$ if
and only if the $n \times n$ matrix
\begin{equation}\label{Eschoenbergmatrix}
A := ( d(x_0, x_j)^2 + d(x_0, x_k)^2 - d(x_j, x_k)^2 )_{j,k=1}^n
\end{equation}
is positive semidefinite of rank $r$.
\end{theorem}

We also refer the reader to \cite{Sch} for more general results for
separable $X$, and \cite{Bo,Mato} for more recent, well-known variants
with constraints on the `embedding dimension' $r$.
\begin{enumerate}
\item We end with some examples and comments in each of the two settings,
starting with $\mathcal{S}([n])$. Note that Proposition~\ref{Ps1n} shows
an isometric embedding into $\R^{\binom{n}{2}}$ with the $1$-norm for a
subset of $\mathcal{S}([n])$. It would be interesting to explore into
what Banach space can the larger subset of norms explored in
Remark~\ref{Rs1n} be isometrically embedded. Note that this is the
restriction of the following metric on the union of the $X,Y$-axes:
\[
d((x,0), (0,y)) := \begin{cases}
\| (x,y) \|_1, \qquad & \text{if } xy \geqslant 0;\\
\| (x,y) \|_\infty, & \text{otherwise},
\end{cases}
\]
and $d$ restricted to the $X$ or $Y$ axis is the usual Euclidean
distance. Can this metric space be (better) understood in terms of
an isometrically embedding into a Banach space?

Another question is if Theorem~\ref{TSnembed} can be strengthened, to
characterize the finite metric spaces on at most $\binom{n}{2}$ elements,
which can be embedded isometrically -- i.e., without scaling the metric
-- into $\mathcal{S}([n])$.

\item Here are some examples of `finite-dimensional embeddings' for
infinite subsets of $\mathcal{S}_k(\F)$. Recall from
Proposition~\ref{Pskp} that for each $p \in [1,\infty]$, the family of
norms $\{ N_{p,q,j} : \ q \in [0,\infty), \ j \in [k] \}$ isometrically
embeds into $\R^k$ with the $1$-norm. Next, by Proposition~\ref{Plp} the
$p$-norms isometrically embed inside a one-dimensional real normed space
(in fact, inside $[0, \log k]$). On the other hand for the $p$-norms, one
can show that the image $\mathcal{S}'_k(\R)$ of the $p$-norms in $C(\R^k
\setminus \{ {\bf 0} \}, \R) / \sim$ (akin to~\eqref{Echoose}) has affine
hull of infinite -- in fact uncountable -- dimension. However, this is a
consequence of the specific embedding and not an intrinsic property of
$\mathcal{S}'_k(\R)$. Thus, it is not clear what is the smallest
(dimensional) Banach space containing an isometric copy of
$\mathcal{S}_k(\F)$.
\end{enumerate}
\end{enumerate}



\end{document}